\documentclass{article}
\usepackage{amsmath}
\usepackage{amssymb}
\usepackage{amsthm}
\usepackage{comment}

\newtheorem{theorem}{Theorem}[section]
\newtheorem{lemma}[theorem]{Lemma}
\newtheorem{proposition}[theorem]{Proposition}

\theoremstyle{definition}
\newtheorem{definition}[theorem]{Definition}
\newtheorem{example}[theorem]{Example}

\theoremstyle{remark}
\newtheorem{remark}[theorem]{Remark}

\numberwithin{equation}{section}

\title{On metric viscosity solutions for Hamilton-Jacobi equations of evolution type}

\author{Atsushi~Nakayasu$^{\mathrm{a}}$}

\date{}

\begin{document}

\maketitle

\begin{center}
$^{\mathrm{a}}$
\textit{Graduate School of Mathematical Sciences, University of Tokyo \\
3-8-1 Komaba, Meguro-ku, Tokyo, 153-8914 Japan,}
\texttt{ankys@ms.u-tokyo.ac.jp}
\end{center}

\begin{abstract}
This paper studies Hamilton-Jacobi equations of evolution type defined in a general metric space.
We give a notion of a solution through optimal principles and establish a unique existence theorem of the solution for initial value problems.
We also note a relationship between the notion of a solution and another notion based on characterization of the modulus of the gradient in the sense of \cite{Giga2012}.
\end{abstract}


\section{Introduction}
\label{s41}

This paper studies Hamilton-Jacobi equations of evolution type defined in a general metric space $(\mathcal{X}, d)$.
One of the most simple problem is the fully nonlinear equation of the form
\begin{equation}
\label{e412}
u_{t}+|\mathit{D}u| = f(x) \quad \text{in $\mathcal{X}\times(0, T)$}
\end{equation}
for a given bounded continuous function $f$ and an unknown function $u = u(x, t)$ on $\mathcal{X}\times(0, T)$ with $T > 0$;
let $u_{t}$ denote the derivative with respect to the time variable $t$
and $|\mathit{D}u|$ formally denote the modulus of gradient in the space variable $x$ with $\mathbf{R}_{+} = [0, \infty)$ values in the sense of \cite{Giga2012}.
That is,
\begin{equation}
\label{e413}
|\mathit{D}u|(x, t) := \sup_{\xi \in \mathit{Lip}^{1}_{x}(\mathcal{X})}\left\{|w_{s}(0)| \mid \text{$w(s, t) = u(\xi(s), t)$}\right\},
\end{equation}
where $\mathit{Lip}^{1}_{x}(\mathcal{X})$ is the set of all absolutely continuous curves $\xi : \mathbf{R} \to \mathcal{X}$ satisfying
$$
|\xi'|(t) := \lim_{s \to t}\frac{d(\xi(s), \xi(t))}{|s-t|} \leq 1 \quad \text{a.e.\ $t$}
$$
and $\xi(0) = x$.
Note that the metric derivative $|\xi'|$ is defined for an absolutely continuous curve $\xi$ in a metric space; see \cite[Chapter 1]{Ambrosio2008}.
However, since a metric space has no tangent space structure in general, the quantity corresponding to the derivative $\xi'$ and the gradient $\mathit{D}u$ is not well-defined.

Hamilton-Jacobi equations are fundamental in various fields of mathematics and physics
and there are many works studying the equation including \eqref{e411}.
We point out that the theory of viscosity solutions is successful for Hamilton-Jacobi equations defined on a Euclidean space,
which is introduced by Crandall and Lions \cite{Lions1982}, \cite{Crandall1983}.
This theory is extended to Banach spaces \cite{Crandall1985}, \cite{Crandall1986}, \cite{Crandall1986a}, \cite{Crandall1990};
this extension is expected to be useful for discussing an optimal control problem with respect to partial differential equations;
we refer the reader to \cite{Feng2009} and itself studies a resolvent problem of Hamilton-Jacobi equations in generalized spaces.
Hamilton-Jacobi equations also appear in the optimal transport theory \cite[Chapter 7, 22, 30]{Villani2009}
and the equation in Wasserstein spaces is studied by Gangbo, Nguyen and Tudorascu \cite{Gangbo2008}.
We also note that study on Hamilton-Jacobi equations on a space with junctions such as a network helps considering the LWR model of traffic flows \cite{Lighthill1955}, \cite{Richards1956}.
Such a problem on a space with junctions is studied in \cite{Imbert2011}, \cite{Imbert2013}.

In order to handle Hamilton-Jacobi equations in such generalized spaces, Giga, Hamamuki and Nakayasu introduced a notion of a viscosity solution of Eikonal equation in \cite{Giga2012}.
We study time evolution equations in the present paper.

We establish a unique existence theorem for an initial value problem of the Hamilton-Jacobi equation.
Consider the value function of an optimal control problem
$$
U(x, t) = \sup_{\xi \in \mathit{Lip}^{1}_{x}(\mathcal{X})}\left\{\int_{0}^{t}f(\xi(r))\mathit{d}r+u_{0}(\xi(t))\right\},
$$
where $u_{0}$ is an bounded uniformly continuous function.
Then, this value function $U$ formally solves the Hamilton-Jacobi equation \eqref{e412} with an initial condition $U(x, 0) = u_{0}(x)$.
It is remarkable that the value function $U$ satisfies
$$
U(x, t) = \sup_{\xi \in \mathit{Lip}^{1}_{x}(\mathcal{X})}\left\{\int_{0}^{h}f(\xi(r))\mathit{d}r+U(\xi(h), t-h)\right\},
$$
which is called a dynamic programming principle; see, e.g.\ \cite{Bardi1997}.

We define a notion of a subsolution and a supersolution based on the dynamic programming principle:
A function $u$ is a subsolution if $w(h) := u(\xi(h), t-h)$ satisfies
$$
-w'(h) \leq f(\xi(h))
$$
in the viscosity sense for all curves $\xi$ called admissible
while $u$ is a supersolution if there exists an  and   such that
$$
-w'(h) \geq f(\xi(h))
$$
in the viscosity sense for some admissible curve $\xi$.
Then, we have the unique existence theorem easily.
However, it is not clear how this definition relates to the original equation \eqref{e412}.
We show that our subsolution is equivalent to a subsolution of \eqref{e412} in the sense of \cite{Giga2012}.
However, it seems to be difficult to show the similar statement for a supersolution.

We point out that closely related topics have studied by Ambrosio and Feng \cite{Ambrosio2013} and of Gangbo and Swiech \cite{Gangbo2014a}, \cite{Gangbo2014b}.
They study the Hamilton-Jacobi equations including \eqref{e411} in a complete geodesic metric space.
However, our theory is applicable to any spaces with metric structure.

This paper is organized as follows.
In Section \ref{s42} we prepare to handle generalized Hamilton-Jacobi equations in a metric space and define a class of admissible curves and a sub- and supersolution based on the dynamic programming.
In Section \ref{s43} we prove some equivalent conditions for a subsolution and a supersolution.
The unique existence theorem will be shown in Section \ref{s44}.

\section{Preliminary and definition of solutions}
\label{s42}

Consider the Hamilton-Jacobi equation of the form
\begin{equation}
\label{e411}
u_{t}+H(x, |\mathit{D}u|) = 0 \quad \text{in $\mathcal{X}\times(0, T)$}
\end{equation}
with a function $H = H(x, p) : \mathcal{X}\times\mathbf{R}_{+} \to \mathbf{R}$ satisfying:
\begin{itemize}
\item[(A1)]
$H = H(x, p)$ is a continuous function in $\mathcal{X}\times\mathbf{R}_{+}$.
\item[(A2)]
$H$ is convex and nondecreasing with respect to the variable $p$ for each $x \in \mathcal{X}$.
\end{itemize}
Define the function
$$
L(x, v) = \sup_{p \in \mathbf{R}_{+}}(p v-H(x, p)) \in \mathbf{R}\cup\{\infty\}
\quad \text{for $(x, v) \in \mathcal{X}\times\mathbf{R}_{+}$.}
$$

We then see that
\begin{proposition}
\label{t421}
Assume (A1) and (A2).
Then, the function $L = L(x, v)$ is lower semicontinuous in $\mathcal{X}\times\mathbf{R}_{+}$ and it is also convex and nondecreasing with respect to the variable $v$ for each $x \in \mathcal{X}$.
In addition, the equations
\begin{align}
\label{e421}
H(x, p) = \sup_{v \in \mathbf{R}_{+}}(p v-L(x, v)) &\quad \text{for all $(x, p) \in \mathcal{X}\times\mathbf{R}_{+}$,} \\
\label{e422}
H(x, |p|) = \sup_{v \in \mathbf{R}}(p v-L(x, |v|)) &\quad \text{for all $(x, p) \in \mathcal{X}\times\mathbf{R}$}
\end{align}
hold.
\end{proposition}

\begin{proof}[Proof]
Since (A1) shows that $(x, v) \mapsto p v-H(x, p)$ is continuous for each $p \in \mathbf{R}_{+}$,
we see that the supremum $L$ is lower semicontinuous.
We also see that $v \mapsto L(x, v)$ is convex and nondecreasing
since $v \mapsto p v-H(x, p)$ is affine and nondecreasing for each $(x, p) \in \mathcal{X}\times\mathbf{R}_{+}$.

Show the equation \eqref{e422}.
First we easily see that
$$
L(x, |v|) = \sup_{p \in \mathbf{R}}(p v-H(x, |p|))
$$
and hence
\begin{equation}
\label{e423}
H(x, |p|) \geq p v-L(x, |v|)
\end{equation}
for all $x \in \mathcal{X}$, $p, v \in \mathbf{R}$.
Therefore, it suffices to show that the equality of \eqref{e423} holds for some $v \in \mathbf{R}$.
Note that $p \mapsto H(x, |p|)$ is convex in $\mathbf{R}$ by (A2).
Thus, for each $p \in \mathbf{R}$ there exists $v \in \mathbf{R}$ such that
$$
H(x, |q|) \geq v(q-p)+H(x, |p|)
\quad \text{for all $q \in \mathbf{R}$,}
$$
which implies
$$
p v-H(x, |p|) \geq \sup_{q \in \mathbf{R}}(v q-H(x, |q|)) = L(x, |v|).
$$
We hence have \eqref{e422}.
The equation \eqref{e421} follows from \eqref{e422}.
\end{proof}

We will also assume:
\begin{itemize}
\item[(A3)]
$c_{H}(p) := \sup_{x \in \mathcal{X}}H(x, p) < \infty$ for each $p \in \mathbf{R}_{+}$.
\item[(A4)]
$\inf_{x \in \mathcal{X}}H(x, 0) > -\infty$ and
$$
\liminf_{p \to \infty}\inf_{x \in \mathcal{X}}\frac{H(x, p)}{p} > 0.
$$
\item[(A5)]
$L$ is continuous on $D_{L} := \{(x, v) \in \mathcal{X}\times\mathbf{R}_{+} \mid L(x, v) < \infty\}$.
A map $x \mapsto V_{L}(x) := \sup\{v \in \mathbf{R}_{+} \mid (x, v) \in D_{L}\}$ is lower semicontinuous on $\mathcal{X}$.
\end{itemize}

\begin{remark}
The condition (A3) implies that
\begin{itemize}
\item[(A3)']
$L(x, v) \geq -c_{H}(0) > -\infty$ for all $(x, v) \in \mathcal{X}\times\mathbf{R}_{+}$ and $\ell(v) := \inf_{x \in \mathcal{X}}L(x, v)$ satisfies
\begin{equation}
\label{e4215}
\liminf_{v \to \infty}\frac{\ell(v)}{v} = \infty.
\end{equation}
\end{itemize}
Indeed, by definition we have
$$
L(x, v) \geq p v-H(x, p) \geq p v-c_{H}(p)
\quad \text{for all $p \in \mathbf{R}_{+}$,}
$$
which shows $L(x, v) \geq -c_{H}(0)$ and $\liminf_{v \to \infty}\ell(v)/v \geq p$.

We also see that the condition (A4) implies that
\begin{itemize}
\item[(A4)']
there exists $V_{L} > 0$ such that
$\sup_{x \in \mathcal{X}}L(x, V_{L}) < \infty$
\end{itemize}
if (A2) holds.
Indeed, note that there exist $V > 0$ and $P \in \mathbf{R}_{+}$ such that $H(x, p) \geq p V$ for all $(x, p) \in \mathcal{X}\times[P, \infty)$ and that $H(x, p) \geq H(x, 0) \geq \inf_{x \in \mathcal{X}}H(x, 0) =: C$.
Hence, we see that
$$
\begin{aligned}
L(x, v)
&= \max\{\sup_{p \leq P}(p v-H(x, p)), \sup_{p \geq P}(p v-H(x, p))\} \\
&\leq \max\{(P v-C), \sup_{p \leq P}p(v-V)\},
\end{aligned}
$$
and so
$
L(x, V) \leq P V-C < \infty.
$

We point out that the assumptions (A3) and (A4) can be replaced by (A3)' and (A4)'
since we need only (A3)' and (A4)' on the main part of this paper.
\end{remark}

\begin{example}
Consider the Hamiltonian of the form
$$
H(x, p) = \sigma(x)h(p)-f(x),
$$
where $h$ is a continuous, convex, nondecreasing, nonconstant function on $\mathbf{R}_{+}$;
$\sigma$ and $f$ are bounded continuous functions on $\mathcal{X}$ with $\inf_{x}\sigma > 0$.
Then the conditions (A1)--(A5) are fulfilled.
The Lagrangian $L$ becomes
$$
L(x, v) = \sigma(x)l\left(\frac{v}{\sigma(x)}\right)+f(x)
$$
with $l(v) = \sup_{p \in \mathbf{R}_{+}}(p v-h(p))$.
Typical examples of such $h$ and $l$ include
$$
h(p) = \frac{1}{2}p^{2},
\quad l(v) = \frac{1}{2}v^{2},
$$
and
$$
h(p) = p,
\quad l(v) =
\begin{cases}
0 & \text{if $v \leq 1$,} \\
\infty & \text{if $v > 1$.} \\
\end{cases}
$$
\end{example}

We introduce a class of trajectories.
\begin{definition}
Let $\mathit{AC}(I, \mathcal{X})$ denote the set of all absolutely continuous curves in $\mathcal{X}$ defined on a interval $I$ of $\mathbf{R}$.

Let $\mathcal{A}(\mathcal{X})$ be the set of all \emph{admissible} curves $\xi \in \mathit{AC}([0, \infty), \mathcal{X})$ such that the metric derivative $|\xi'|$ is piecewise constant and $L[\xi]$ is piecewise continuous;
$|\xi'|$ equals a constant $v_{I}$ and $r \mapsto L(\xi(r), v_{I})$ is continuous on each $I = [0, r_{1}], [r_{1}, r_{2}], \cdots, [r_{n}, \infty)$ with finitely many $r_{1}, \cdots, r_{n}$.
Let
$
\mathcal{A}_{x}(\mathcal{X}) = \{\xi \in \mathcal{A}(\mathcal{X}) \mid \xi(0) = x\}
$
for $x \in \mathcal{X}$.
\end{definition}

\begin{remark}
\begin{enumerate}
\item
Consider a constant curve $\xi(r) = x$ for a fixed point $x \in \mathcal{X}$.
Since $|\xi'| = 0$, we see that $\mathcal{A}_{x}(\mathcal{X})$ is nonempty for all $x \in \mathcal{X}$.
\item
For each $\xi \in \mathit{AC}(\mathbf{R}, \mathcal{X})$ with $\xi(0) = x$ take $\hat{\xi} \in \mathit{AC}((L^{-}, L^{+}), \mathcal{X})$ in the next proposition.
Set $\tilde{\xi}(r) = \hat{\xi}(V_{L}r)$ for $0 \leq r \leq L^{+}/V_{L}$ and $\tilde{\xi}(r) = \hat{\xi}(L_{+})$ for $r \geq L^{+}/V_{L}$ with $V_{L} > 0$ in (A4)'.
Then, since $|\hat{\xi}'| = V_{L}$ a.e.\ on $[0, L^{+}/V_{L})$, (A4)' and (A5) imply $\tilde{\xi} \in \mathcal{A}_{x}(\mathcal{X})$.
\end{enumerate}
\end{remark}

\begin{proposition}
\label{t425}
For $\xi \in \mathit{AC}(\mathbf{R}, \mathcal{X})$ set
\begin{equation}
\label{e429}
\tau_{\xi}(h) = \int_{0}^{h}|\xi'|\mathit{d}r
\quad \text{for $h \in \mathbf{R}$.}
\end{equation}
Then, there exists a curve $\hat{\xi} \in \mathit{AC}((L^{-}, L^{+}), \mathcal{X})$ such that
\begin{equation}
\label{e4210}
\xi = \hat{\xi}\circ\tau_{\xi}, \quad \text{$|\hat{\xi}'| = 1$ a.e.\ in $(L^{-}, L^{+})$}
\end{equation}
with $L^{\pm} := \lim_{h \to \pm\infty}\tau_{\xi}(h)$.
\end{proposition}

This is a well-known fact on the absolutely continuous curves.
We refer the reader to \cite[Lemma 1.1.4]{Ambrosio2008} for its proof.

In order to define a notion of a solution, we recall a notion of a superdifferential and a subdifferential in the viscosity sense.
For a continuous function $w$ defined on an open set $W$ in $\mathbf{R}^{N}$ define the \emph{superdifferential} $\mathit{D}^{+}w(x)$ and the \emph{subdifferential} $\mathit{D}^{-}w(x)$ at $x \in W$ as below:
\begin{align*}
\mathit{D}^{+}w(x) &:= \{\mathit{D}\varphi(x) \mid \text{$\varphi$ is a $\mathit{C}^{1}$ supertangent of $w$ at $x$}\}, \\
\mathit{D}^{-}w(x) &:= \{\mathit{D}\varphi(x) \mid \text{$\varphi$ is a $\mathit{C}^{1}$ subtangent of $w$ at $x$}\},
\end{align*}
where we say that $\varphi$ is a \emph{$\mathit{C}^{1}$ supertangent} (resp.\ \emph{subtangent}) of $w$ at $x$
if there exists a neighborhood $U \subset W$ of $x$ such that $\varphi \in \mathit{C}^{1}(U)$ and
$$
\max_{U}(w-\varphi) = (w-\varphi)(x). \quad \text{(resp.\ $\min_{U}(w-\varphi) = (w-\varphi)(x)$.)}
$$
As an analogue we define a suitable set of a superdifferential and a subdifferential for a piecewise continuous function $w$ defined on an interval $I$ in $\mathbf{R}$ at $h \in I$:
Set
\begin{align*}
\mathit{D}^{+, r}w(h) &:= \{\varphi'(h+0) \mid \text{$\varphi$ is a piecewise $\mathit{C}^{1}$ right supertangent of $w$ at $h$}\}, \\
\mathit{D}^{-, r}w(h) &:= \{\varphi'(h+0) \mid \text{$\varphi$ is a piecewise $\mathit{C}^{1}$ right subtangent of $w$ at $h$}\},
\end{align*}
where we say that $\varphi$ is a \emph{piecewise $\mathit{C}^{1}$ right supertangent} (resp.\ \emph{subtangent}) of $w$ at $h$
if there exists $r > 0$ such that $\varphi$ is piecewise $\mathit{C}^{1}$ on $[h, h+r)$ and
$$
\max_{[h, h+r)}(w-\varphi) = (w-\varphi)(h). \quad \text{(resp.\ $\min_{[h, h+r)}(w-\varphi) = (w-\varphi)(h)$.)}
$$

We define a notion of a subsolution and a supersolution of the equation \eqref{e411}.
Let $\mathcal{Q} := \mathcal{X}\times(0, T)$.
\begin{definition}
\label{d426}
Let $u$ be an arcwise continuous function in $\mathcal{Q}$;
for every $\xi \in \mathit{AC}(\mathbf{R}, \mathcal{X})$ the function $w(s, t) = u(\xi(s), t)$ is continuous in $\mathbf{R}\times(0, T)$.

We call $u$ a \emph{subsolution} of \eqref{e411}
if for each $(x, t) \in \mathcal{Q}$ and every $\xi \in \mathcal{A}_{x}(\mathcal{X})$
the inequality
\begin{equation}
\label{e425}
-p \leq L(\xi(h), |\xi'|(h+0))
\end{equation}
holds for all $p \in \mathit{D}^{+, r}w(h)$ and all $h \in [0, t)$.

We call $u$ a \emph{supersolution} of \eqref{e411}
if for each $(x, t) \in \mathcal{Q}$ and $\varepsilon > 0$
there exist $\xi \in \mathcal{A}_{x}(\mathcal{X})$ and a continuous function $w$ such that
\begin{equation}
\label{e426}
w(0) = u(x, t),
\quad w(h) \geq u(\xi(h), t-h)-\varepsilon
\end{equation}
and the inequality
\begin{equation}
\label{e427}
-p \geq L(\xi(h), |\xi'|(h+0))
\end{equation}
holds for all $p \in \mathit{D}^{-, r}w(h)$ and all $h \in [0, t)$.
\end{definition}

\section{Remarks on the solution}
\label{s43}

The definition of a subsolution and a supersolution is based on a sub- and superoptimality principle.
For simplicity write
$$
L[\xi](r) := L(\xi(r), |\xi'|(r)).
$$
The following propositions are valid:

\begin{proposition}
\label{t427}
For an arcwise continuous function $u$ on $\mathcal{Q}$ the following conditions are equivalent:
\begin{itemize}
\item[(i)]
$u$ is a subsolution of \eqref{e411}.
\item[(ii)]
$u$ satisfies a suboptimality:
For each $(x, t) \in \mathcal{Q}$ and each $\xi \in \mathcal{A}_{x}(\mathcal{X})$
the inequality
\begin{equation}
\label{e428}
u(x, t) \leq \int_{0}^{h}L[\xi]\mathit{d}r+u(\xi(h), t-h)
\quad \text{for all $h \in [0, t)$}
\end{equation}
holds.
\end{itemize}
\end{proposition}

\begin{proposition}
\label{t428}
For an arcwise continuous function $u$ on $\mathcal{Q}$ the following conditions are equivalent:
\begin{itemize}
\item[(i)]
$u$ is a supersolution of \eqref{e411}.
\item[(ii)]
$u$ satisfies a superoptimality:
For each $(x, t) \in \mathcal{Q}$ and $\varepsilon > 0$
there exists $\xi \in \mathcal{A}_{x}(\mathcal{X})$
such that the inequality
\begin{equation}
\label{e424}
u(x, t) \geq \int_{0}^{h}L[\xi]\mathit{d}r+u(\xi(h), t-h)-\varepsilon
\quad \text{for all $h \in [0, t)$}
\end{equation}
holds.
\end{itemize}
\end{proposition}

\begin{proof}[Proof of Proposition \ref{t427}]
First show $\text{(ii)}\Rightarrow\text{(i)}$.
Fix $(x, t) \in \mathcal{Q}$, $\xi \in \mathcal{A}_{x}(\mathcal{X})$ and a piecewise $\mathit{C}^{1}$ right subtangent of $w(h) := u(\xi(h), t-h)$ at $h \in [0, h)$.
Note that by considering $\tilde{\xi}(r) = \xi(r+h)$ in the suboptimality we have
$$
\begin{aligned}
w(h) = u(\xi(h), t-h)
&\leq \int_{0}^{\theta}L[\tilde{\xi}]\mathit{d}r+u(\tilde{\xi}(\theta), t-h-\theta) \\
&= \int_{h}^{h+\theta}L[\xi]\mathit{d}r+w(h+\theta)
\end{aligned}
$$
for all $\theta \in [0, t-h)$.
We hence obtain
$$
\varphi(h)-\varphi(h+\theta)
\leq w(h)-w(h+\theta)
\leq \int_{h}^{h+\theta}L[\xi]\mathit{d}r,
$$
which implies \eqref{e425} with $p = \varphi(h+0)$ since $L[\xi]$ is piecewise continuous.

Next show $\text{(i)}\Rightarrow\text{(ii)}$.
Fix $(x, t) \in \mathcal{Q}$ and $\xi \in \mathcal{A}_{x}(\mathcal{X})$, and let
with $w(h) := u(\xi(h), t-h)$.
Note that $\int_{0}^{h}L[\xi]\mathit{d}r$ is piecewise $\mathit{C}^{1}$ on $[0, t)$ since $L[\xi]$ is piecewise continuous.
Therefore,
$$
\ell(h) := w(h)+\int_{0}^{h}L[\xi]\mathit{d}r
$$
satisfies $p \geq 0$ for all $p \in \mathit{D}^{+, r}\ell(h)$ and all $h \in [0, t)$ since \eqref{e425} holds for all
$$
p \in \mathit{D}^{+, r}w(h) = \mathit{D}^{+, r}\ell(h)-L(\xi(h), |\xi'|(h+0))
$$
We now claim that $\ell(0) \leq \ell(h)$ for each $h \in [0, t)$.
Suppose, on the contrary, that $\ell(0) > \ell(h)$ at some $h \in (0, t)$.
Since $\ell(0) > \ell(0)-c > \ell(h)$ holds for some positive number $c$,
the function
$$
\ell(r)+\frac{c}{h}r
$$
attains a maximum over $[0, h]$ at some $r^{*} \in [0, h)$.
Hence, we have $-c \geq 0$, which contradicts to $c > 0$.
Therefore, we see that $\ell(0) \leq \ell(h)$ and so \eqref{e428} holds for all $h \in [0, t)$.
\end{proof}

\begin{proof}[Proof of Proposition \ref{t428}]
First show $\text{(ii)}\Rightarrow\text{(i)}$.
For $(x, t) \in \mathcal{Q}$ and $\varepsilon > 0$ take $\xi \in \mathcal{A}_{x}(\mathcal{X})$ such that \eqref{e424} holds.
Set
$$
w(h) := u(x, t)-\int_{0}^{h}L[\xi]\mathit{d}r.
$$
so that \eqref{e426} holds.
Since $L[\xi]$ is piecewise continuous,
$w$ is piecewise $\mathit{C}^{1}$ and
$$
w'(h+0) = -L(\xi(h), |\xi'|(h+0))
\quad \text{for all $h \in [0, t)$,}
$$
which shows that $u$ is a supersolution.
Indeed, for each piecewise $\mathit{C}^{1}$ right subtangent $\varphi$ of $w$ at some $h \in [0, t)$,
we see that
$$
\varphi(h)-\varphi(h+\theta)
\geq w(h)-w(h+\theta)
= \int_{h}^{h+\theta}L[\xi]\mathit{d}r
$$
for all $\theta > 0$ small enough,
which implies \eqref{e427} with $p = \varphi(h+0)$.

Next show $\text{(i)}\Rightarrow\text{(ii)}$.
For each $(x, t) \in \mathcal{Q}$ and $\varepsilon > 0$
take $\xi \in \mathcal{A}_{x}(\mathcal{X})$ and a continuous function $w$ such that \eqref{e426} and \eqref{e427} hold.
Note that $\int_{0}^{h}L[\xi]\mathit{d}r$ is piecewise $\mathit{C}^{1}$ on $[0, t)$ since $L[\xi]$ is piecewise continuous.
Therefore,
$$
\ell(h) := w(h)+\int_{0}^{h}L[\xi]\mathit{d}r
$$
satisfies $p \leq 0$ for all $p \in \mathit{D}^{-, r}\ell(h)$ and all $h \in [0, t)$ since \eqref{e427} holds for all
$$
p \in \mathit{D}^{-, r}w(h) = \mathit{D}^{-, r}\ell(h)-L(\xi(h), |\xi'|(h+0))
$$
We now claim that $\ell(0) \geq \ell(h)$ for each $h \in [0, t)$.
Suppose, on the contrary, that $\ell(0) < \ell(h)$ at some $h \in (0, t)$.
Since $\ell(0) < \ell(0)+c < \ell(h)$ holds for some positive number $c$,
the function
$$
\ell(r)-\frac{c}{h}r
$$
attains a minimum over $[0, h]$ at some $r^{*} \in [0, h)$.
Hence, we have $c \leq 0$, which contradicts $c > 0$.
Therefore, we have $\ell(0) \geq \ell(h)$ so that
$$
w(0) \geq w(h)+\int_{0}^{h}L[\xi]\mathit{d}r,
$$
which yields \eqref{e424} by \eqref{e426} for all $h \in [0, t)$.
\end{proof}

Next let us consider relationship between a solution by Definition \ref{d426} and another one based on the characterization of the modulus of gradient \eqref{e413}.
Set
$$
\mathit{LC}^{1}_{x}(\mathcal{X}) = \{\xi \in \mathit{AC}([0, \infty), \mathcal{X}) \mid \text{$|\xi'| \leq 1$, $\xi(0) = x$, $|\xi'|$ is piecewise constant}\}.
$$

\begin{definition}[Metric viscosity solution]
Let $u$ be an arcwise continuous function on $\mathcal{Q}$.

We call $u$ a \emph{metric viscosity subsolution} of \eqref{e411}
if for each $(x, t) \in \mathcal{Q}$ and every $\xi \in \mathit{LC}^{1}_{x}(\mathcal{X})$
the inequality
\begin{equation}
\label{e431}
q+H(x, |p|) \leq 0
\end{equation}
holds for all $(p, q) \in \mathit{D}_{s, t}^{+}w(0, t)$ with $w(s, t) = u(\xi(s), t)$,
where
$$
\mathit{D}_{s, t}^{+}w(s, t) := \{(\varphi_{s}, \varphi_{t})(s, t) \mid \text{$\varphi$ is a $\mathit{C}^{1}$ supertangent of $w$ at $(s, t)$}\}.
$$
\end{definition}

We then have

\begin{proposition}
Assume (A1)--(A5) and let $u$ be an arcwise continuous function on $\mathcal{Q}$.
Then, the following conditions are equivalent:
\begin{itemize}
\item[(i)]
$u$ is a subsolution of \eqref{e411}.
\item[(ii)]
$u$ is a metric viscosity subsolution of \eqref{e411}.
\item[(iii)]
$u$ satisfies a suboptimality.
\end{itemize}
\end{proposition}

\begin{proof}[Proof]
The statement $\text{(i)}\Leftrightarrow\text{(iii)}$ has already shown in Proposition \ref{t428}.

Show $\text{(iii)}\Rightarrow\text{(ii)}$.
Fix $(x, t) \in \mathcal{Q}$, $\xi \in \mathit{LC}^{1}_{x}(\mathbf{R}, \mathcal{X})$ and a $\mathit{C}^{1}$ supertangent $\varphi$ of $w(s, t) = u(\xi(s), t)$ at $(0, t)$.
In order to prove \eqref{e431}
we should show
\begin{equation}
\label{e432}
q+|p|v-L(x, v) \leq 0
\end{equation}
for all $v \in \mathbf{R}_{+}$.
Note that \eqref{e432} is trivial for $v > V_{L}(x)$, i.e.\ $L(x, v) = \infty$
and we only need to show \eqref{e432} for all $v < V_{L}(x)$
since letting $v \to V_{L}(x)$ yields \eqref{e432} at $v = V_{L}(x)$.
Take $\sigma(r) = \pm v r$ so that $\xi\circ\sigma$ is $v$-Lipschitz.
We now observe that
$$
\begin{aligned}
\varphi(0, t)-\varphi(\sigma(h), t-h)
&\leq u(x, t)-u(\xi(\sigma(h)), t-h) \\
&\leq \int_{0}^{h}L[\xi\circ\sigma]\mathit{d}r \\
&\leq \int_{0}^{h}L(\xi(\sigma(r)), v)\mathit{d}r
\end{aligned}
$$
for all $h \in [0, t)$ small enough.
Since $v < V_{L}(x)$ it follows from (A5) that $r \in [0, t) \mapsto L(\xi(\sigma(r)), v)$ is continuous at $0$.
Therefore, we have
$$
-p\sigma'(0)+q \leq L(x, v)
$$
which yields \eqref{e432} and so \eqref{e431} holds.

Next show $\text{(ii)}\Rightarrow\text{(iii)}$.
First note that for any $\hat{\xi} \in \mathit{LC}^{1}_{x}(\mathcal{X})$ the function $w(s, t) = u(\hat{\xi}(s), t)$ satisfies
\begin{equation}
\label{e435}
q+H(\xi(s), |p|)
= \sup_{v \in \mathbf{R}}\{q+p v-L(\xi(s), |v|)\}
\leq 0
\end{equation}
for all $(p, q) \in \mathit{D}_{s, t}^{+}w(s, t)$ and all $(s, t) \in \mathbf{R}\times(0, T)$.
Fix $(\hat{x}, \hat{t}) \in \mathcal{Q}$ and $\xi \in \mathcal{A}_{\hat{x}}(\mathcal{X})$.
Take $\hat{\xi} \in \mathit{LC}^{1}_{x}(\mathcal{X})$ and $\tau = \tau_{\xi}$ satisfying \eqref{e429}, \eqref{e4210}.
We show that
\begin{equation}
\label{e436}
w(0, \hat{t}) \leq w(\tau(h), \hat{t}-h)+\int_{0}^{h}L(\hat{\xi}(\tau(r)), \tau'(r))\mathit{d}r
\quad \text{for all $h \in [0, \hat{t})$}
\end{equation}
Note that $\tau'$ is piecewise constant and $L[\xi\circ\tau]$ is piecewise continuous.
Let us take an interval $I = (a, b)$ with $[a, b] \subset [0, \hat{t})$ on which $\tau'$ is constant $v \in \mathbf{R}_{+}$ and $r \mapsto L(\hat{\xi}(\tau(r)), v)$ is continuous.
By \eqref{e435}
$$
q-p v-L(\hat{\xi}(s), v) \leq 0
\quad \text{for all $(p, q) \in \mathit{D}^{+}_{s, t}w(s, t)$, $(s, t) \in J\times(0, T)$,}
$$
where $J = \tau(I)$ if $v > 0$ or $J = \mathbf{R}$ if $v = 0$.
By a classical result on viscosity solutions,
$$
w(s, t) \leq w(\lambda(h))+\int_{0}^{h}L(\hat{\xi}(s+v r), v)\mathit{d}r
$$
for all $(s, t) \in J\times(0, T)$ and $0 \leq h < h^{*} := \sup\{h \in [0, \infty) \mid \lambda(h) \in J\times(0, T)\}$
with $\lambda(h) = (s+v h, t-h)$.
Note that $h^{*} \geq (b-a)\wedge t$.
Letting $s \to \tau(a)$, $t = \hat{t}-a$, $h \to (b-a)\wedge t = b-a$,
we have
$$
w(\tau(a), \hat{t}-a) \leq w(\tau(b), \hat{t}-b)+\int_{a}^{b}L(\hat{\xi}(\tau(r)), v)\mathit{d}r.
$$
Combining such an inequality shows \eqref{e436}, which implies \eqref{e428}.
\end{proof}

\section{Unique existence theorem}
\label{s44}

We study the initial value problem of \eqref{e411} with
\begin{equation}
\label{e441}
u|_{t = 0} = u_{0} \quad \text{on $\mathcal{X}$,}
\end{equation}
where we assume that
\begin{itemize}
\item[(A6)]
$u_{0}$ is bounded and uniformly continuous on $\mathcal{X}$;
there exists a modulus $\omega_{u_{0}}$ such that
$$
|u_{0}(x)-u_{0}(y)| \leq \omega_{u_{0}}(d(x, y)) \quad \text{for all $x, y \in \mathcal{X}$.}
$$
\end{itemize}
Here, a modulus $\omega$ is a function of the class $\mathit{C}(\mathbf{R}_{+}, \mathbf{R}_{+})$ with $\omega(0) = 0$.

The main purpose of this section is to establish a unique existence theorem for \eqref{e411} and \eqref{e441}.
\begin{definition}
An arcwise continuous function $u$ on $\mathcal{X}\times[0, T)$ is called a \emph{solution} of the initial value problem \eqref{e411} and \eqref{e441}
if $u$ is both a subsolution and a supersolution of \eqref{e411}, and satisfies
\begin{equation}
\label{e4412}
u(0, x) = u_{0}(x) \quad \text{for all $x \in \mathcal{X}$.}
\end{equation}
\end{definition}

Consider the \emph{cost functional} to the initial value problem \eqref{e411} and \eqref{e441};
$$
C_{t}[\xi] := \int_{0}^{t}L[\xi]\mathit{d}r+u_{0}(\xi(t))
$$
for $t \in [0, T]$ and $\xi \in \mathcal{A}_{x}(\mathcal{X})$.
Define the \emph{value function} $U$ by
$$
U(x, t) = \sup_{\xi \in \mathcal{A}_{x}(\mathcal{X})}C_{t}[\xi]
\quad \text{for $(x, t) \in \mathcal{X}\times[0, T]$.}
$$


We will show that the value function is a unique solution of \eqref{e411} and \eqref{e441}.

\begin{theorem}
\label{t442}
Assume (A1)--(A6).
Then the value function $U$ is a unique solution of \eqref{e411} and \eqref{e441}.
\end{theorem}

We first show a regularity of the value function.

\begin{lemma}
\label{t441}
Assume (A1)--(A6).
Then the value function $U$ satisfies
\begin{equation}
\label{e4411}
-c_{H}(0)t+\inf u_{0} \leq U(x, t) \leq t L(x, 0)+u_{0}(x)
\quad \text{for all $(x, t) \in \mathcal{X}\times[0, T]$.}
\end{equation}
\end{lemma}

\begin{proof}[Proof]
Let $(x, t) \in \mathcal{X}\times[0, T]$.
Since $\xi_{0}(r) =x$ is of $\mathcal{A}_{x}(\mathcal{X})$,
we have
$$
C_{t}[\xi_{0}] = t L(x, 0)+u_{0}(x) < \infty.
$$
For each $\xi \in \mathcal{A}_{x}(\mathcal{X})$, the assumptions imply that
$$
C_{t}[\xi]
= \int_{0}^{t}L[\xi]\mathit{d}r+u_{0}(\xi(h))
\geq -c_{H}(0)t+\inf u_{0}
> -\infty.
$$
Therefore, we have \eqref{e4411}.
\end{proof}

\begin{proposition}
\label{t444}
Assume (A1)--(A6).
Then the value function $U$ is bounded and arcwise uniformly continuous on $\mathcal{X}\times[0, T]$, i.e.\ for each $\xi \in \mathit{Lip}^{1}(\mathcal{X})$ the function $w(s, t) = u(\xi(s), t)$ is uniformly continuous in $\mathbf{R}\times[0, T]$.
\end{proposition}

\begin{proof}[Proof]
First note that \eqref{e4411} implies
\begin{equation}
\label{e442}
-|c_{H}(0)|T+\inf u_{0} \leq U(x, t) \leq T|\sup L(x, 0)|+\sup u_{0}.
\end{equation}
In particular, $U$ is a bounded function.

Fix $\xi \in \mathit{AC}(\mathcal{X})$.
In order to show continuity of $(s, t) \mapsto U(\xi(s), t)$ let us estimate $U(\xi(s), t)-U(\xi(\bar{s}), \bar{t})$ for $s, \bar{s} \in \mathbf{R}$, $t, \bar{t} \in [0, T]$.
Let $\varepsilon > 0$.
By the definition of $U(\xi(\bar{s}), \bar{t})$ there exists a curve $\bar{\xi} \in \mathcal{A}_{x}(\mathcal{X})$ such that
\begin{equation}
\label{e443}
U(\xi(\bar{s}), \bar{t})
\geq \int_{0}^{\bar{t}}L[\bar{\xi}]\mathit{d}r+u_{0}(\bar{\xi}(\bar{t}))-\varepsilon.
\end{equation}
We now construct a curve $\tilde{\xi} \in \mathcal{A}_{x}(\mathcal{X})$ such that
$$
\tilde{\xi}(r) = \bar{\xi}(r-r_{0}) \quad \text{for $r \geq r_{0}$}
$$
with some $r_{0} \geq 0$.
Note that for such a curve
\begin{equation}
\label{e444}
\begin{aligned}
U(\xi(s), t)
&\leq \int_{0}^{t}L[\tilde{\xi}]\mathit{d}r+u_{0}(\tilde{\xi}(t)) \\
&= \int_{0}^{r_{0}}L[\tilde{\xi}]\mathit{d}r+\int_{0}^{t-r_{0}}L[\bar{\xi}]\mathit{d}r+u_{0}(\bar{\xi}(t-r_{0})).
\end{aligned}
\end{equation}
Set
$$
\tilde{\xi}(r) =
\begin{cases}
\hat{\xi}(\tau(s)+V_{L}\frac{\tau(\bar{s})-\tau(s)}{|\tau(\bar{s})-\tau(s)|}r) & \text{for $0 \leq r \leq |\tau(\bar{s})-\tau(s)|/V_{L} =: r_{1}$} \\
\xi(\bar{s}) & \text{for $r_{1} \leq r \leq r_{1}+|\bar{t}-t| =: r_{0}$} \\
\end{cases}
$$
with $\hat{\xi}$ and $\tau = \tau_{\xi}$ taken by Proposition \ref{t425}.
Then, noting that $|\tilde{\xi}'| = 0$ on $[0, r_{1}]$ and $|\tilde{\xi}'| = V_{L}$ on $[r_{1}, r_{0}]$,
we have
\begin{equation}
\label{e445}
R_{1} := \int_{0}^{r_{0}}L[\tilde{\xi}]\mathit{d}r
\leq \int_{0}^{r_{0}}L(\tilde{\xi}(r), V_{L})\mathit{d}r
\leq r_{0}\sup_{x}L(x, V_{L})
\end{equation}
The inequalities \eqref{e443}--\eqref{e445} yields
\begin{equation}
\label{e446}
U(\xi(s), t)-U(\xi(\bar{s}), \bar{t})
\leq R_{1}-R_{2}+R_{3}+\varepsilon,
\end{equation}
where
$$
R_{2} := \int_{t-r_{0}}^{\bar{t}}L[\bar{\xi}]\mathit{d}r,
\quad R_{3} := u_{0}(\bar{\xi}(t-r_{0}))-u_{0}(\bar{\xi}(\bar{t})).
$$
Now noting that
$$
t-r_{0} = t-|\tau(\bar{s})-\tau(s)|/V_{L}-|\bar{t}-t| \leq \bar{t},
$$
we hence see that
\begin{equation}
\label{e447}
R_{2}
\geq -c_{H}(\bar{t}-t+r_{0}).
\end{equation}
We also have
\begin{equation}
\label{e448}
R_{3} \leq \omega_{u_{0}}\left(\int_{t-r_{0}}^{\bar{t}}|\bar{\xi}'|\mathit{d}r\right).
\end{equation}
holds in any cases.
Combining \eqref{e445}--\eqref{e448}, we have
$$
\begin{aligned}
U(\xi(s), t)-U(\xi(\bar{s}), \bar{t}) \leq
&r_{0}\sup_{x}L(x, V_{L})+c_{H}(0)(\bar{t}-t+r_{0}) \\
&+\omega_{u_{0}}\left(\int_{t-r_{0}}^{\bar{t}}|\bar{\xi}'|\mathit{d}r\right)+\varepsilon.
\end{aligned}
$$
Note that
$$
r_{1} = |\tau(\bar{s})-\tau(s)|/V_{L} = \frac{1}{V_{L}}\left|\int_{s}^{\bar{s}}|\xi'|\mathit{d}r\right|
\leq \frac{|\bar{s}-s|}{V_{L}} \to 0
$$
as $|\bar{s}-s| \to 0$.
Therefore, if we show
\begin{equation}
\label{e4410}
\lim_{\delta \downarrow 0}\sup_{|\bar{s}-s|+|\bar{t}-t| \leq \delta}\int_{t-r_{0}}^{\bar{t}}|\bar{\xi}'|\mathit{d}r
= 0,
\end{equation}
the proof is completed.

By \eqref{e442} and \eqref{e446} we observe that
$$
\begin{aligned}
&\int_{t-r_{0}}^{\bar{t}}L(\bar{\xi}(r), |\bar{\xi}'|(r))\mathit{d}r \\
&\quad \leq r_{0}\sup_{x}L(x, V_{L})+R_{3}+\varepsilon-U(\xi(s), t)+U(\xi(\bar{s}), \bar{t}) \\
&\quad \leq 4\sup|g|+1+(\mathrm{Lip}[\xi]/V_{L}+1)|\sup_{x}L(x, V_{L})|+T|\sup_{x}L(x, 0)|+|c_{H}(0)|T,
\end{aligned}
$$
for $\varepsilon < 1$, $|\bar{s}-s|+|\bar{t}-t| < 1$ and hence
$$
\int_{t-r_{0}}^{\bar{t}}\ell(|\bar{\xi}'|)\mathit{d}r \leq C < \infty
$$
holds with some constant $C$ independent of $\bar{s}, s, \bar{t}, t$.
Therefore, the next lemma shows \eqref{e4410}.
\end{proof}

\begin{lemma}
Let $\{v_{n}\}$ be a sequence of nonnegative, Lebesgue measurable functions $v_{n}$ defined on $[0, s_{n}]$ with $s_{n} \downarrow 0$ such that
$$
\int_{0}^{s_{n}}\ell(v_{n}(r))\mathit{d}r \leq C \quad \text{for $n \in \mathbf{N}$}
$$
holds for some constant $C$.
Then,
$$
\int_{0}^{s_{n}}v_{n}(r)\mathit{d}r \to 0 \quad \text{as $n \to \infty$.}
$$
\end{lemma}

\begin{proof}[Proof]
Note that \eqref{e4215} implies that for every large $M > 0$ there exists $V_{M} \geq 0$ such that $\ell(v) \geq M v$ holds for all $v \leq V_{M}$.
We now observe that
$$
\begin{aligned}
\int_{0}^{s_{n}}\ell(v_{n}(r))\mathit{d}r
&= \int_{[0, s_{n}]\cap\{v_{n} \geq V_{M}\}}\ell(v_{n}(r))\mathit{d}r+\int_{[0, s_{n}]\cap\{v_{n} < V_{M}\}}\ell(v_{n}(r))\mathit{d}r \\
&\geq M\int_{[0, s_{n}]\cap\{v_{n} \geq V_{M}\}}v_{n}(r)\mathit{d}r-c_{H}(0)s_{n}.
\end{aligned}
$$
We also see by this that
$$
\begin{aligned}
\int_{0}^{s_{n}}v_{n}(r)\mathit{d}r
&= \int_{[0, s_{n}]\cap\{v_{n} \geq V_{M}\}}v_{n}(r)\mathit{d}r+\int_{[0, s_{n}]\cap\{v_{n} < V_{M}\}}v_{n}(r)\mathit{d}r \\
&\leq \frac{1}{M}\int_{[0, s_{n}]\cap\{v_{n} \geq V_{M}\}}\ell(v_{n}(r))\mathit{d}r+\frac{C_{H}}{M}s_{n}+V_{M}s_{n} \\
&\leq \frac{C}{M}+\frac{C_{H}}{M}s_{n}+V_{M}s_{n}.
\end{aligned}
$$
Therefore,
$$
\limsup_{n \to \infty}\int_{0}^{s_{n}}v_{n}(r)\mathit{d}r \leq \frac{C}{M}
$$
for all $M > 0$,
and so letting $M \to \infty$ implies the conclusion.
\end{proof}

\begin{theorem}
\label{t446}
Assume (A1)--(A6).
Then the value function $U$ is a solution of \eqref{e411}, \eqref{e441}.
\end{theorem}

\begin{proof}[Proof]
First show that $U$ is a subsolution.
Fix $(x, t) \in \mathcal{Q}$ and $\xi \in \mathcal{A}_{x}(\mathcal{X})$.
Since there exists $\bar{\xi} \in \mathcal{A}_{x}(\mathcal{X})$ for $h \in [0, t]$ and $\varepsilon > 0$ such that
$$
U(\xi(h), t-h)
\geq \int_{0}^{t-h}L[\bar{\xi}]\mathit{d}r+u_{0}(\bar{\xi}(t-h))-\varepsilon,
$$
taking $\tilde{\xi}(r) = \xi(r)$ for $0 \leq r \leq h$, $\tilde{\xi}(r) = \bar{\xi}(r-h)$ for $r \geq h$ implies that
$$
\begin{aligned}
U(x, t)
&\leq \int_{0}^{t}L[\tilde{\xi}]\mathit{d}r+u_{0}(\tilde{\xi}(t)) \\
&= \int_{0}^{h}L[\xi]\mathit{d}r+\int_{0}^{t-h}L[\bar{\xi}]\mathit{d}r+u_{0}(\bar{\xi}(t-h)).
\end{aligned}
$$
Combining these two inequalities yields
$$
U(x, t) \leq \int_{0}^{h}L[\xi]\mathit{d}r+U(\xi(h), t-h)+\varepsilon.
$$
Since $\varepsilon$ is arbitrary, $U$ satisfies a suboptimality and hence $U$ is a subsolution by Proposition \ref{t427}.

To prove that $U$ is a supersolution, for $(x, t) \in \mathcal{Q}$ and $\varepsilon > 0$ take $\xi \in \mathcal{A}_{x}(\mathcal{X})$ such that
$$
U(x, t) \geq \int_{0}^{t}L[\xi]\mathit{d}r+u_{0}(\xi(t))-\varepsilon.
$$
Since $\tilde{\xi}(r) = \xi(r+h)$ belongs to $\mathcal{A}_{x}(\mathcal{X})$ for $h \in [0, t]$,
we have
$$
U(\xi(h), t-h) \leq \int_{h}^{t}L[\xi]\mathit{d}r+u_{0}(\xi(t)).
$$
Combining these two inequalities yields
$$
U(x, t) \geq \int_{0}^{h}L[\xi]\mathit{d}r+U(\xi(h), t-h)-\varepsilon.
$$
Therefore, $U$ satisfies a superoptimality and hence $U$ is a supersolution by Proposition \ref{t428}.

Since it is clear that $U$ satisfies \eqref{e4412} by definition,
we see that $U$ is a solution.
\end{proof}

\begin{remark}
This proof also shows that a dynamic programming principle is valid for the value function:
$$
U(x, t) = \inf_{\xi \in \mathcal{A}_{x}(\mathcal{X})}\left\{\int_{0}^{h}L[\xi]\mathit{d}r+U(\xi(h), t-h)\right\} \quad \text{for all $h \in [0, t]$.}
$$
This condition also indicates that the value function satisfies a semigroup property.
\end{remark}

We show a comparison theorem.

\begin{theorem}
\label{t448}
Assume (A1)--(A6).
Assume that arcwise continuous functions $u$ and $v$ on $\mathcal{X}\times[0, T)$ are a subsolution and a supersolution, respectively.
Then, the inequality
\begin{equation}
\label{e4413}
\sup_{\mathcal{Q}}(u-v) \leq \sup_{x \in \mathcal{X}}(u(x, 0)-v(x, 0))
\end{equation}
holds.
\end{theorem}

\begin{proof}[Proof]
Fix $(x, t) \in \mathcal{Q}$, $\varepsilon > 0$ and $\xi \in \mathcal{A}_{x}(\mathcal{X})$ such that
$$
v(x, t) \geq \int_{0}^{h}L[\xi]\mathit{d}r+v(\xi(h))-\varepsilon
\quad \text{for all $h \in [0, t]$.}
$$
Note that $L[\xi]$ is piecewise continuous on $[0, t]$.
Since $(s, t) \mapsto v(\xi(s), t)$ is continuous on $[0, t]\times[0, T)$ by the arcwise continuity of $v$,
letting $h \to t$ we have
$$
v(x, t) \geq \int_{0}^{t}L[\xi]\mathit{d}r+v(\xi(t), 0)-\varepsilon
$$
We also see that
$$
u(x, t) \leq \int_{0}^{h}L[\xi]\mathit{d}r+u(\xi(h), t-h)
\quad \text{for all $h \in [0, t)$,}
$$
which implies
$$
u(x, t) \leq \int_{0}^{t}L[\xi]\mathit{d}r+u(\xi(t), 0).
$$
Combining the inequalities, we obtain
$$
\begin{aligned}
u(x, t) - v(x, t)
&\leq u(\xi(t), 0)-v(\xi(t), 0)+\varepsilon \\
&\leq \sup_{x \in \mathcal{X}}(u(x, 0)-v(x, 0))+\varepsilon.
\end{aligned}
$$
Since $\varepsilon$ is arbitrary, the proof is complete.
\end{proof}

\begin{proof}[Proof of Theorem \ref{t442}]
Proposition \ref{t444} and Theorem \ref{t446} imply that the value function $U$ is a solution.
If $u$ is another solution, Theorem \ref{t448} implies that $u \leq U$ and $U \leq u$ in $\mathcal{Q}$ since $u|_{t = 0} = U|_{t = 0} = u_{0}$, and hence we see that $u = U$.
\end{proof}

\section*{Acknowledgments}
The author thank Yoshikazu~Giga for his constructive comments and continuous encouragement.
The work of the author was supported by a Grant-in-Aid for JSPS Fellows No.\ 25-7077 and the Program for Leading Graduate Schools, MEXT,  Japan.


\providecommand{\bysame}{\leavevmode\hbox to3em{\hrulefill}\thinspace}
\providecommand{\MR}{\relax\ifhmode\unskip\space\fi MR }
\providecommand{\MRhref}[2]{%
  \href{http://www.ams.org/mathscinet-getitem?mr=#1}{#2}
}
\providecommand{\href}[2]{#2}

\end{document}